\documentclass[12pt]{article}
 
\usepackage{amssymb,amsmath}
\usepackage{amsthm,bbm}
\usepackage{hyperref}
\usepackage{parskip}

\newtheorem{theorem}{Theorem}[section]
\newtheorem{proposition}[theorem]{Proposition}

\theoremstyle{remark}
\newtheorem{remark}{Remark}

\usepackage[
  backend=biber,
  style=numeric,
  sorting=nyt,
  maxbibnames=10,
  citestyle=numeric-comp
]{biblatex}
\usepackage{authblk}
 
\begin{document}
 \title{Well-posedness of stochastic reacting particle systems\\ with non-local and Lennard--Jones interactions}
\author[a]{Daniela Morale }
\author[a]{Giulia Rui }
\author[a]{Stefania Ugolini }

\affil[a]{\small Department of Mathematics, University of Milano, Italy\\ \texttt{Daniela.Morale,Giulia.Rui, Stefania.Ugolini@unimi.it}}

\maketitle

\begin{abstract}
We establish well-posedness results for systems of a finite number of stochastic particles driven by independent Brownian motions and subject to a strongly singular drift induced by a Lennard–Jones interaction. In addition to the pairwise force, the dynamics includes a nonlocal drift mediated by an environmental field, whose evolution is coupled to the particle configuration through a regularized empirical density. We then extend the analysis to a reaction model in which the switching (or killing) rate also depends on the field. An interlacing technique is considered for establishing the well-posedness of the full system. The model is motivated by the challenge to provide a stochastic microscopic description of the sulphation phenomenon in cultural heritage materials.
\end{abstract}
 


\smallskip

\section{Introduction}

The aim of the paper is to investigate the well-posedness of a system consisting of finitely many interacting stochastic particles whose dynamics is governed by three main components: a strongly singular pairwise interaction, an additional nonlocal drift generated by an evolving environmental field that is itself influenced by the particle configuration, and a reaction mechanism that removes particles from the active dynamics, with a field-dependent hazard intensity.

More precisely, let $d>1$; for $x\in\mathbb R^d\setminus\{0\}$, we consider the Lennard--Jones type potential
\begin{equation}\label{eq:V_LJ}
  V(x) := \frac{A}{|x|^\alpha} - \frac{B}{|x|^\beta},
  \qquad A,B>0,\quad \alpha>\beta>0,
\end{equation}
and the associated force $F(x):=-\nabla V(x)$. Such a potential presents a strong repulsive singularity of $V$ at close range, preventing particle overlap, as well as a weak attractive interaction at intermediate distances.

Despite its very frequent use in applied modelling and molecular dynamics \cite{2005_Nedea_molecular_dynamics,Wales2024}, this combination of properties places the model largely outside classical well-posedness results for SDEs with singular drifts, typically based on local integrability assumptions. To the best of our knowledge, mathematical literature offers few results on this type of force, largely confined to modelling approaches \cite{2023_FLR,2016_MZCJ,Rui_MACH_2021,javergaard2026hybrid}. A cornerstone reference in that direction is the paper by \cite{2005_krylov_rockner}, which, under deterministic initial conditions, provides non-explosion criteria in the special case of gradient drifts with potentials blowing up at the boundary of a set. See also \cite{2015bogachev}. 

In contrast with regularisation-by-noise results, in the present setting
the singular repulsive drift itself provides the stabilising mechanism:
the divergence of the Lennard--Jones potential near collisions acts as a
barrier preventing particle collapse.

A related feedback mechanism appears in
\cite{HamblyLedgerSojmark2019}, where particles are absorbed upon hitting
a barrier and the resulting loss feeds back onto the dynamics of the
remaining particles. Although the killing mechanism is different from
the one considered here, both models feature a feedback effect generated
by particle removal.

Our analysis builds on the framework developed in \cite{rui2025LJ}.

Let $\left(\Omega,\mathcal{F},\mathbb F=(\mathcal{F}_{t})_{t\in [0,T]},\mathbb{P}\right)$ be a filtered probability space satisfying the usual conditions. We fix a finite time horizon $T>0$ and an integer $N\ge 2$, and let $\{(W_{t}^{1},..., W_{t}^{N})_{t\in[0,T]}\}$ be a family of independent $d$-dimensional $\mathbb F$-adapted Wiener processes. 

To describe particles that have undergone a reaction, we introduce a cemetery state \(\Delta\) and set $D:=\mathbb R^d\cup\{\Delta\}$.
The state of the particle system is described by the stochastic process
\begin{equation*}
 (X,H)= \left\{(X^i,H^i), i=1,...,N \right\}=\left\{\left(X^i_t,H^i_t \right)_{t\in[0,T]}, i=1,...,N \right\},
\end{equation*}
where $X^i_t\in D$ is the position of the $i$-th particle and $H^i_t\in\mathbb H:=\{0,1\}$ is its activity label.

Throughout the paper, spatial expressions multiplied by activity indicators are understood as being evaluated only on active coordinates.
More precisely, we use the convention
\[
    \sum_{i=1}^N
    \mathbbm 1_{\{0\}}(H_t^i)F(X_t^i)
    :=
    \sum_{\{i:H_t^i=0\}} F(X_t^i),
\]
and analogously for sums involving several particle indices.

Let \(\tau_i\) denote the reaction time of particle \(i\), which will be
constructed below as an \(\mathbb F\)-stopping time. The particle is active before $\tau_i$, and it is removed from the spatial dynamics after reacting
\[
X_t^i\in\mathbb R^d, \ \text{for}\, t<\tau_i,
\qquad
X_t^i\in \{\Delta\}, \ \text{for} \,t\ge\tau_i.
\]
The activity label \(H_t^i\in\{0,1\}\) records whether particle \(i\)
has reacted: \(H_t^i=0\) while the particle is active, whereas \(H_t^i=1\) after the reaction. The process \(H^i\), together with the
reaction time \(\tau_i\), will be constructed below through a Poisson
random measure. By construction,
\begin{equation*}
 H_t^i = \mathbbm{1}_{[\tau_i,\infty)}(t),\qquad t\in[0,T],
\end{equation*}
and
\[
(X_t^i,H_t^i)=(\Delta,1),
\qquad t\geq\tau_i.
\]
The spatial distribution of the active particles is described by the empirical measure $\nu_t$ and its mollification \(u_N\), or \emph{empirical regularised density}
\begin{align*}
  \nu_t^N &:= \frac{1}{N}\sum_{i=1}^N \mathbbm 1_{\{0\}}(H_{t}^i)\,\,\varepsilon_{X_t^i}(\cdot), & t\in[0,T] \nonumber\\
  u_N(t,x)&:=(K*\nu_t^N)(x),& (t,x)\in[0,T]\times\mathbb R^d, 
\end{align*}
where $\varepsilon_x$ is the delta-Dirac measure localising in $x\in \mathbb R^d$ and $K\in C_b^\infty(\mathbb R^d)$ is a smooth non-negative mollifier.

The particle distribution is coupled to a possibly random environmental
concentration
\[
    c:\Omega\times[0,T]\times\mathbb R^d\to\mathbb R_+,
\]
where the dependence on \(\omega\) is suppressed throughout.
Its evolution is governed by the mass-action equation
\begin{equation}\label{eq:calcium_evolution}
\partial_t c(t,x) = -\lambda\,c(t,x)\,u_N(t,x),
\qquad (t,x)\in(0,T]\times\mathbb R^d,
 \end{equation}
with initial condition $c(0,x)=c_0(x)$.
We assume that the chosen representative of \(c_0\) satisfies
\[
    0\le c_0(\omega,x)\le m_0(x)
\]
for almost every \(\omega\) and every \(x\in\mathbb R^d\), where $m_0:\mathbb R^d\to(0,\infty)$ is a bounded reference profile representing the total amount of material initially available at position $x\in \mathbb R^d$.

Since \eqref{eq:calcium_evolution} is an ODE at each spatial point, throughout the paper we identify \(c\) with the
measurable representative
\[
    c(t,x)
    =
    c_0(x)
    \exp\left(
        -\lambda\int_0^t u_N(s,x)\,ds
    \right).
\]
Then, whenever \(c(t,\cdot)\) is regarded as an element of \(L^2(\mathbb R^d)\), pointwise expressions such as \(c(t,X_t^i)\)
always refer to this representative.
For a random initial field, the same convention is applied pathwise to
the jointly measurable representative of \(c_0\).

We set \(\overline m:=\|m_0\|_\infty\) and assume that there exists
\(\underline m>0\) such that
\begin{equation}\label{eq:belowbound_m0}
  \underline m  \le m_0(x) \le \overline{m}, \qquad  x\in\mathbb R^d .
\end{equation}
\begin{remark}
The uniform lower bound in \eqref{eq:belowbound_m0} is imposed for notational convenience. 
All the results of the paper remain valid, with minimal and straightforward modifications, if $m_0$ is merely nonnegative, bounded, and supported on a measurable set $S\subset\mathbb R^d$
of positive Lebesgue measure, provided that
\[
 m_0(x) \ge \underline m >0, \qquad x\in S.
\]
This includes the hybrid model considered in
\cite{javergaard2026hybrid}. In that setting, the environmental interaction is understood to act only on $S$ by restricting the integrals to such domain. All subsequent results are unchanged, and the proofs require only this notational modification.
\end{remark}

The underlying field $c$ influences the particle dynamics in two ways. First, it biases the motion of active particles via a non-local drift. For \(x\in\mathbb R^d\) and \(h\in\mathbb H\), we define
\begin{equation}\label{eq:F_env}
  \overline{G}[c(t,\cdot)](x,h)
  :=
  \mathbbm{1}_{\{0\}}(h)\int_{B_R(x)}\frac{z-x}{|z-x|}
  \left[1-\frac{c(t,z)}{m_0(z)}\right]
  e^{-|z-x|}\,dz.
\end{equation}
We set \(\overline G[c](\Delta,h):=0\) for \(h\in\mathbb H\). At the same time, the integrand in \eqref{eq:F_env} is set equal to zero at \(z=x\). Here, $R>0$ is a fixed interaction radius, and the factor $\bigl(1-c/m_0\bigr)$ measures the local depletion of the environmental field and biases active particles toward
regions in which this depletion is larger. The factor \(\mathbbm{1}_{\{0\}}(h)\) ensures that the environmental
drift acts only on active particles.

Because of representation \eqref{eq:calcium_evolution}, the drift \(\overline{G}[c(t,\cdot)]\) is itself a functional of the previous evolution of the particle system.

Second, the environmental field $c$ also determines the rate at which active particles react. 
For each \(i=1,\ldots,N\), while particle \(i\) is active, we define the
state-dependent reaction intensity
\begin{equation*}%\label{eq:stochastic_intensity_i}
\lambda_i(t)
:=
\widetilde{\lambda}\,c(t,X_{t^-}^i),
\qquad
\widetilde{\lambda}\geq0.
\end{equation*}
Let
\[
{M^i(ds,dz)},
\qquad i=1,\ldots,N,
\]
be independent Poisson random measures on
\((0,T]\times\mathbb R_+\), independent of the Brownian motions, with
intensity measure {\(ds\otimes dz\)}. The activity process is defined by
\[
H_t^i
=
H_0^i
+
\int_{(0,t]\times\mathbb R_+}
\mathbbm 1_{\{0\}}(H_{s^-}^i)
\mathbbm 1_{\{z\leq\lambda_i(s)\}}
\,M^i(ds,dz),
\qquad t\in[0,T].
\]
Thus, a point \((s,z)\) of \(M^i\) is accepted as a reaction event whenever particle \(i\) is still active immediately before time \(s\) and \(z\leq\lambda_i(s)\). This is the standard thinning construction for a counting process with stochastic intensity; see, for instance, \cite{PointProcesses_GunterBrandt,Applebaum}.
Thus \(H^i\) has at most one jump, from \(0\) to \(1\). We define the
reaction time of particle \(i\) by
\begin{equation}\label{eq:tau_def}
\tau_i
:=
\inf\{t>0:H_t^i=1\},
\end{equation}
with the convention \(\inf\emptyset:=+\infty\).

Accordingly,
\[
X_t^i\in\mathbb R^d,\qquad t<\tau_i,
\]
whereas
\[
(X_t^i,H_t^i)=(\Delta,1),
\qquad t\geq\tau_i.
\]
Equivalently, the reaction time may be described through its cumulative
hazard
\[
\Lambda_i(t)
=
\widetilde{\lambda}
\int_0^t c(s,X_s^i)\,ds,
\qquad t<\tau_i.
\]
Equivalently, on an enlarged probability space one may introduce an
independent exponential random variable \(\zeta_i\sim\mathrm{Exp}(1)\)
and write
\[
\tau_i
=
\inf\{t\geq0:\Lambda_i(t)\geq\zeta_i\}.
\]
We can now state the complete particle–field dynamics. As in \cite{rui2025LJ}, active particles interact through the singular Lennard–Jones potential \(V\)
\begin{equation}\label{eq:LJ}
        V (x) := \frac{A}{|x|^\alpha} - \frac{B}{|x|^\beta}, \qquad \alpha>\beta> 0.
\end{equation}
We retain the convention that self-interactions are excluded when such singular interaction is involved.

For every $i\in N^*=\{1,\dots,N\}$, the dynamics reads:
\begin{equation}\label{eq:full_system}
\begin{split}
  H_t^i &= H_0^i+ \int_{(0,t]\times \mathbb R_+}
\mathbbm 1_{\{0\}}(H_{s^-}^i)\,
\mathbbm 1_{\left\{z\le \lambda_i(s)\right\}}
\,M^i(ds,dz), \qquad t>0\\
dX_t^i
  &= -\nabla V * \nu_t^N (X_t^i)\,dt 
   + \overline{G} [c(t,\cdot)](X_t^i,H_t^i)\,dt
   + \sigma\, \,dW_t^i, \quad \qquad t <\tau_i;\\
X_t^i & = \Delta, \hspace{8.8cm} t\in [\tau_i,T],
\end{split}
\end{equation} 
coupled with the evolution of the underlying field $c$ \eqref{eq:calcium_evolution}.

Since the Lennard--Jones interaction is pairwise, self-interactions are
excluded throughout. More precisely, whenever the interaction convolution
is evaluated at the position of particle \(i\), we use the slight abuse of
notation
\[
\nabla V * \nu_t^N(X_t^i)
:=
\frac{1}{N}
\sum_{\substack{j=1\\ j\neq i}}^N
\mathbbm 1_{\{0\}}(H_t^j)
\nabla V(X_t^i-X_t^j).
\]
The same convention is used for the regularised potentials
\(V_\varepsilon\) and for the auxiliary empirical measures introduced
below.

\section{Well-posedness with a fixed number of active particles}

The first step in the analysis of the full system \eqref{eq:full_system}
is to establish existence and pathwise uniqueness for a finite number of interacting particles 
coupled with an environmental field, in the absence of any reaction or switching mechanism. 
We therefore begin by considering the hybrid system \eqref{eq:full_system}
without any killing rate: the parameter $\widetilde{\lambda}=0$ and all Brownian particles remain
active for any $t\in [0,T].$ 
 
After relabelling if necessary, let
\(1\le\widetilde N\le N\) denote the number of active particles. System \eqref{eq:full_system}
reduces, for any $t\in [0,T]$, to 
\begin{equation}\label{eq:SDE_LJcg_nojumps} 
\begin{split}
dX_t^i &=
-\nabla V * \widetilde{\nu}_t^{\widetilde{N}} (X_t^i)\,dt
+ \widetilde{G}[c(t,\cdot)](X_t^i)\,dt
+ \sigma\,dW_t^i, \qquad i=1,\dots,\widetilde N,\\[0.35cm]
\partial_t c(t,x) & = -\lambda\,c(t,x)\,\widetilde u_{\widetilde N}(t,x),
\hspace{5.3cm}\qquad x\in \mathbb R^d,\\
\end{split}
\end{equation}
with initial conditions given by an $\mathcal{F}_0$-measurable random vector $X_0=(X_0^1,\dots,X_0^{\widetilde N})$ and
 $c(0,x)=c_0(x)$, such that $0\le c_0(x)\le m_0(x)$, for the environmental field \(c\). 
 Again, $V$ is the Lennard--Jones potential 
 \eqref{eq:LJ} while the environmental drift is $\widetilde{G}[c(t,\cdot)](x) := \overline{G}[c(t,\cdot)](x,0)$ with $\overline{G}$ as in \eqref{eq:F_env}.
 
 The measure $\widetilde{\nu}_t^{\widetilde{N}}$ in \eqref{eq:SDE_LJcg_nojumps} is a counting measure rescaled by $N$, defined as 
\begin{equation}\label{eq:nu_tilde}
\widetilde{\nu}_t^{\widetilde{N}}:=\frac{1}{{N}}\sum_{i=1}^{\widetilde N}\varepsilon_{X_t^i}, \quad t\in [0,T]
\end{equation}
We define the corresponding regularised density by
\[
\widetilde u_{\widetilde N}(t,\cdot)
:=
K*\widetilde\nu_t^{\widetilde N}.
\]
The distinction will be crucial in Section~\ref{sec:final}, where the number of active particles decreases over time due to reactions.

While the Lennard--Jones force term is strongly singular at the origin, the non-local drift term is sufficiently regular.
\begin{proposition}[Environmental regularity]\label{prop:regularity_G}
Let $d\ge 1$ and let $m_0:\mathbb R^d\to(0,\infty)$ be defined as in \eqref{eq:belowbound_m0} and set $B_{m_0}(\mathbb R^d):=\{v:\mathbb R^d\to\mathbb R_+ \;:\; 0\le v(x)\le m_0(x)\}.$

For $v\in B_{m_0}(\mathbb R^d)$, we define the vector field
\begin{equation*}
\widetilde G[v](x)
:=\int_{B_R(x)} \frac{z-x}{|z-x|}
\left(1-\frac{v(z)}{m_0(z)}\right)e^{-|z-x|}\,dz ,
\qquad x\in\mathbb R^d.
\end{equation*}
Then there exist constants $C,L_1,L_2<\infty$, depending only on $(d,R,\underline m)$, such that for all $v,v'\in B_{m_0}(\mathbb R^d)$ and all
$x,y\in\mathbb R^d$:
\begin{align}
|\widetilde G[v](x)| &\le C, \label{eq:G_bounded}\\
|\widetilde G[v](x)-\widetilde G[v'](x)|
&\le L_1\,\|v-v'\|_{L^\infty(\mathbb R^d)}, \label{eq:G_Lipschitz_v}\\
|\widetilde G[v](x)-\widetilde G[v](y)|
&\le L_2\,|x-y|. \label{eq:G_Lipschitz_x}
\end{align}
In particular, the map $x\mapsto \widetilde G[v](x)$ is globally Lipschitz, uniformly in
$v\in B_{m_0}(\mathbb R^d)$.
\end{proposition}
\begin{proof}
Let $ f(x):=1-{v(x)}/{m_0(x)}, $ for any $x\in \mathbb R^d$, and $ \phi(x):=\frac{x}{|x|}e^{-|x|}, $ for any $x\in \mathbb R^d\setminus \{0\}.$ Then $|f(x)|\le 1$ and $|\phi(w)| \le 1$.

Given $v\in B_{m_0}(\mathbb R^d)$, this trivially gives the uniform bound \eqref{eq:G_bounded}
\[
\left|\widetilde G[v](x)\right| \le  |B_R(0)| =: C.
\]

For any $v,v'\in B_{m_0}(\mathbb R^d)$,
\[
|\widetilde G[v](x)-\widetilde G[v'](x)|
\le \int_{B_R(x)} \frac{|v(z)-v'(z)|}{m_0(z)} e^{-|z-x|}\,dz
\le \frac{|B_R(0)|}{\underline m}\,\|v-v'\|_{L^\infty},
\]
which yields the Lipschitz continuity with respect to $v$ property \eqref{eq:G_Lipschitz_v} with $L_1=|B_R|/\underline m$. 

For \(h\in\mathbb R^d\), define $\tau_h g(x)=g(x+h)$, and given \(x,y\in\mathbb R^d\), set \(h:=x-y\). Then, for any $v\in B_{m_0}$ and $x,y\in \mathbb R^d$
\begin{equation}\label{eq:proof_L_x}
\begin{split}
\left| \widetilde G[v](x)-\widetilde G[v](y)\right| 
& =
\left|\int_{\mathbb R^d}
f(z)\Big((\mathbbm{1}_{B_R(0)}\phi)(z-x)-\tau_h(\mathbbm{1}_{B_R(0)}\phi)(z-x)\Big)\,dz\right|\\
&\le
\|\mathbbm{1}_{B_R(0)}\phi-\tau_h(\mathbbm{1}_{B_R(0)}\phi)\|_{L^1(\mathbb R^d)}\\
&\le
\|\phi-\tau_h\phi\|_{L^1(\mathbb R^d)}
+
\|\mathbbm{1}_{B_R(0)}-\tau_h \mathbbm{1}_{B_R(0)}\|_{L^1}.
\end{split}
\end{equation}
In order to obtain \eqref{eq:G_Lipschitz_x}, it therefore suffices to prove
\[
\|\phi-\tau_h\phi\|_{L^1(\mathbb R^d)}
\le C|h|, \qquad \|\mathbbm{1}_{B_R(0)}-\tau_h \mathbbm{1}_{B_R(0)}\|_{L^1}\le C|h|.
\]
If $d\ge2$, a direct computation shows $\nabla\phi\in L^1(\mathbb R^d)$, hence $\phi\in W^{1,1}(\mathbb R^d)$. Standard translation estimates for $W^{1,1}$ functions give
\begin{equation}\label{est:Step1}
  \|\phi-\tau_h\phi\|_{L^1(\mathbb R^d)}
\le |h|\,\|\nabla\phi\|_{L^1}.
\end{equation}
If $d=1$, the function $\phi(x)=\mathrm{sign}(x)e^{-|x|}$ satisfies by direct computation
\[
\|\phi-\tau_h\phi\|_{L^1(\mathbb R)}
\le C|h|.
\]
Since \(B=B_R(0)\) is smooth and bounded,
\[
\mathbbm{1}_{B_R(0)}\in BV(\mathbb R^d),
\qquad
|D\mathbbm{1}_{B_R(0)}|(\mathbb R^d)
=
\mathcal H^{d-1}(\partial B)
=
|\mathbb S^{d-1}|R^{d-1},
\]
where $D \mathbbm{1}_{B_R(0)}$ denotes the distributional derivative \cite{EvansGariepy}. Moreover, there exists a sequence
$f_k\in C^\infty(\mathbb R^d)\cap BV(\mathbb R^d)$ such that
\begin{equation}\label{eq:limit_evans}
  f_k\to \mathbbm{1}_{B_R(0)} \qquad\text{in } L^1(\mathbb R^d),
\qquad
|Df_k|(\mathbb R^d)\to |D \mathbbm{1}_{B_R(0)}|(\mathbb R^d).
\end{equation}
see \cite[Theorem 5.3]{EvansGariepy}. Now, for each smooth $f_k$ and $d\geq 2$, we can apply estimate \eqref{est:Step1} and obtain
\[
\|\tau_h f_k-f_k\|_{L^1}\le |h|\int_{\mathbb R^d}|\nabla f_k|\,dz.
\]
A similar estimate holds for $d=1$, by replacing $\nabla f_k$ with $f'_k$ and using the fundamental theorem of calculus in place of Step 1.

Moreover, for smooth functions $|Df_k|(\mathbb R^d)=\int|\nabla f_k|\,dx$ (\cite{EvansGariepy}, Example for $W^{1,1}$ functions).
Thus
\[
\|\tau_h f_k-f_k\|_{L^1}\le |h|\,|Df_k|(\mathbb R^d).
\]
Passing to the limit $k\to\infty$ using \eqref{eq:limit_evans} yields
\begin{equation}\label{eq:BVshift}
\|\tau_h  \mathbbm{1}_{B_R(0)}- \mathbbm{1}_{B_R(0)}\|_{L^1}\le |h|\,|D \mathbbm{1}_{B_R(0)}|(\mathbb R^d).
\end{equation}
From \eqref{eq:proof_L_x}, \eqref{est:Step1}, and \eqref{eq:BVshift}, estimate \eqref{eq:G_Lipschitz_x} is achieved.
\end{proof}
\begin{remark}
If $v$ is globally Lipschitz and $m_0$ is Lipschitz and uniformly positive, then the proof of Proposition \ref{prop:regularity_G} simplifies considerably. Indeed, in this case the environmental drift $\widetilde G[v]$ is globally Lipschitz in $x$ by a direct differentiation argument,
and estimates \eqref{eq:G_bounded}–\eqref{eq:G_Lipschitz_x}
follow immediately from standard convolution bounds.
\end{remark}
The following theorem establishes the well-posedness of the particle system without reaction. It follows the framework developed in \cite{rui2025LJ} for systems with dominant singular gradient drift only. In particular, we extend the application to Lennard--Jones in \cite[Theorem~3.3]{rui2025LJ}, to include  the bounded environmental interaction and the coupling with the field $c$. Note that in \cite[Theorem~3.3]{rui2025LJ} the measure \eqref{eq:nu_tilde} is defined with $\widetilde{N}=N$; however, this does not affect the proof, since all the estimates always concern an arbitrary finite number of particles: the multiplication by a different factor doesn't alter the asymptotic analysis required for the estimates.
\begin{theorem}[Well-posedness of the $\widetilde N$-particle system]\label{thm:finite_wp}
Let $d\ge1$, \(1\le\widetilde N\le N\), and $T>0$.
Assume that the initial condition $X_0=(X_0^1,\dots,X_0^{\widetilde N})$ is $\mathcal{F}_0$-measurable and satisfies
\begin{equation}\label{eq:init_finite}
\mathbb E\bigl[|X_0|^2\bigr]<\infty,
\qquad \qquad
\mathbb E\!\left[ \sum_{1\le i<j \le \widetilde N}|
 V(X_0^i-X_0^j)|
\right]<\infty.
\end{equation}
Let
\(c_0:\Omega\times\mathbb R^d\to\mathbb R_+\)
be \(\mathcal F_0\otimes\mathcal B(\mathbb R^d)\)-measurable, with
\[
    \mathbb E\!\left[
        \|c_0\|_{L^2(\mathbb R^d)}^2
    \right]<\infty,
\]
and assume that
\[
    0\le c_0(x)\le m_0(x)
\]
for almost every \(\omega\) and every \(x\in\mathbb R^d\).
Then the system \eqref{eq:SDE_LJcg_nojumps} admits a pathwise unique strong solution on $[0,T]$.

Moreover, with probability one, no collisions between particles occur on $[0,T]$, i.e.
\[
\mathbb P\Bigl(
\exists\,t\in[0,T],\ \exists\,i\neq j
\ \text{such that } X_t^i=X_t^j
\Bigr)=0.
\]
In particular, the singularity is almost surely never reached in finite time.
Moreover,
\[
    c\in C([0,T];L^2(\mathbb R^d))
    \qquad\text{a.s.},
    \qquad
    0\le c(t,x)\le m_0(x)
\]
for every \(t\in[0,T]\) and \(x\in\mathbb R^d\), for the representative fixed above.
For every bounded stopping time \(\tau\le T\), \(X_\tau\) satisfies
\eqref{eq:init_finite}, while \(c(\tau,\cdot)\) is
\(\mathcal F_\tau\)-measurable as an \(L^2(\mathbb R^d)\)-valued random variable and
\[
    \mathbb E\|c(\tau,\cdot)\|_{L^2}^2<\infty.
\]
Hence the hypotheses of the theorem are preserved at bounded stopping times.
\end{theorem}
\begin{proof}
If \(\widetilde N=1\), the Lennard--Jones term is absent and the result follows directly from the regularity of the environmental drift and the fixed-point argument below. We therefore assume \(\widetilde N\ge2\) in the singular-interaction argument.
In such case, the proof follows two main steps: the study of the well-posedness of a regularised system and the extension of the result to the original system.

\emph{Smoothing the interaction and well-posedness of the regularised system.} 
To apply \cite[Theorem~3.3]{rui2025LJ}, we first introduce a regularised version of the system by smoothing the Lennard--Jones potential near its singularity at the origin. We then pass from the regularised dynamics to the original singular system by using the same non-attainment and localisation argument developed there, adapted here to the additional environmental drift.

Fix $\varepsilon>0$ and let $V_\varepsilon\in C^2(\mathbb R^d)$ be one the regularisation constructed in  \cite[Appendix A]{rui2025LJ}, such that
\[
V_\varepsilon(x)=V(x)\quad\text{for all }x\in \mathbb R^d\setminus B_\varepsilon(0),
\qquad
\nabla V_\varepsilon\in C_b^1(\mathbb R^d;\mathbb R^d).
\]
In particular, $\nabla V_\varepsilon$ is bounded and globally Lipschitz, and
\[
\|\nabla^2 V_\varepsilon\|_{L^\infty(\mathbb R^d)}<\infty.
\]
Moreover, the regularisation preserves the lower bound of the Lennard--Jones potential:
\[
V_\varepsilon(x)\geq-\delta',
\qquad x\in\mathbb R^d,
\]
where \(V\geq-\delta'\), and since $\nabla V_\varepsilon$ is bounded, $V_\varepsilon$ is globally Lipschitz on $\mathbb R^d$.

We denote by $X^{i,\varepsilon}$ the position of the $i$-th particle evolving according to the regularised system, by $\widetilde{\nu}^{\widetilde N,\varepsilon}$ the associated empirical measure, and by
$\widetilde u_{N,\varepsilon}$ the corresponding mollified density. The environmental field has the explicit representation
\begin{equation}\label{eq:c_epsilon}
c_\varepsilon(t,x) =
c_0(x)\exp\!\Bigl(- \lambda\int_0^t \widetilde u_{N,\varepsilon}(s,x)\,ds\Bigr)
= \mathcal{D}\!\left(c_0(x),\int_0^t \widetilde u_{N,\varepsilon}(s,x)\,ds\right),
\end{equation}
for all $(t,x)\in[0,T]\times\mathbb R^d$, where
$\mathcal{D}(a_1,a_2):=a_1\exp(-\lambda a_2)$ for $a_1,a_2\in\mathbb R_+$.

Equation \eqref{eq:c_epsilon} ensures uniform boundedness and regularity of the underlying field, namely that $0\le c_\varepsilon\le \|m_0\|_\infty$ and $c_\varepsilon \in H^1(0,T; L^2(\Omega\times \mathbb R^d))$.
On the other hand, the explicit representation of $c_\varepsilon$ allows us to reduce the coupled system \eqref{eq:SDE_LJcg_nojumps} to a closed equation for the particle process alone.
More precisely, for any $t\in[0,T]$ and $i=1,\dots,\widetilde N$, the particle dynamics can be written as the following path-dependent SDE:
\begin{equation}\label{eq:SDE_jumps_path_depependent} 
\begin{split}
dX_t^{i,\varepsilon} &=\left[-\nabla V_\varepsilon * \widetilde{\nu}_t^{\widetilde N,\varepsilon} + \widetilde{G}\left[\mathcal{D}\left(c_0(\cdot), \int_0^t \widetilde u_{N,\varepsilon}(s,\cdot) ds\right) \right] \right](X_t^{i,\varepsilon})\,dt 
+ \sigma\,dW_t^i.
\end{split}
\end{equation}

Equation~\eqref{eq:SDE_jumps_path_depependent} is path-dependent through the non-local drift, which depends on the past trajectory via the time integral of $\widetilde u_{N,\varepsilon}$. Nevertheless, this formulation is convenient for proving strong existence and pathwise
uniqueness of the regularised system
$\bigl(X_t^{i,\varepsilon},\, i=1,\ldots,\widetilde{N},\, c_t^\varepsilon\bigr)$, since it allows the use of a contraction argument; see, e.g., \cite{morale2025killed,Flandoli_Leocata_2019}.

Fix $\varepsilon>0$ and let $T'<T$. Define the Banach space $(E, \| \cdot \|_E) $
\[
E:=L^2_P\!\left(\Omega;C\!\left([0,T'];\mathbb R^{\widetilde Nd}\right)\right),
\qquad
\|Y\|_E:=
\left(
\mathbb E\!\left[
\sup_{t\in[0,T']}
|Y_t|^2
\right]
\right)^{1/2}.
\]
For \(Y\in E\), define
\[
\nu_t^{N,Y}
:=
\frac1N\sum_{j=1}^{\widetilde N}\delta_{Y_t^j},
\qquad
u_N^Y(t,\cdot)
:=
K*\nu_t^{N,Y},
\]
and
\[
c^Y(t,x)
:=
c_0(x)
\exp\left(
-\lambda\int_0^t u_N^Y(s,x)\,ds
\right),
\qquad
(t,x)\in[0,T']\times\mathbb R^d.
\]
We define a map \(\Psi:E\to E\) componentwise by
\[
[\Psi(Y)]_t^i
:=
X_0^i
+
\int_0^t
\left[
-\nabla V_\varepsilon*\nu_s^{N,Y}(Y_s^i)
+
\widetilde G[c^Y(s,\cdot)](Y_s^i)
\right]\,ds
+
\sigma W_t^i,
\]
for \(i=1,\ldots,\widetilde N\) and \(t\in[0,T']\). {When \(\nabla V_\varepsilon*\nu_t^{N,Y}\) is evaluated at \(Y_t^i\),
the self-interaction \(j=i\) is omitted according to the convention
introduced above.}

By construction, $\nabla V_\varepsilon$ is globally Lipschitz and bounded, while the
environmental drift $\widetilde G$ is bounded and Lipschitz by
Proposition~\ref{prop:regularity_G}. As a consequence, $\Psi(E)\subset E$.

Moreover, for any \(Y,Y'\in E\), the Lipschitz continuity of \(\nabla V_\varepsilon\) gives
\[
\left|
\nabla V_\varepsilon * \nu_s^{N,Y}(Y_s^i)
-
\nabla V_\varepsilon * \nu_s^{N,Y'}(Y_s^{\prime i})
\right|
\le
C_\varepsilon |Y_s-Y'_s|.
\]
On the other hand, since \(K\) is globally Lipschitz,
\[
\|u_N^Y(s,\cdot)-u_N^{Y'}(s,\cdot)\|_{L^\infty}
\le
C|Y_s-Y'_s|,
\]
and therefore, using the explicit representation of the environmental field,
\[
\|c^Y(s,\cdot)-c^{Y'}(s,\cdot)\|_{L^\infty}
\le
C\int_0^s |Y_r-Y'_r|\,dr.
\]
Combining this estimate with Proposition~\ref{prop:regularity_G}, we obtain
\[
\left|
\widetilde G[c^Y(s,\cdot)](Y_s^i)
-
\widetilde G[c^{Y'}(s,\cdot)](Y_s^{\prime i})
\right|
\le
C\left(
|Y_s-Y'_s|
+
\int_0^s |Y_r-Y'_r|\,dr
\right).
\]
Consequently, standard estimates and Cauchy--Schwarz inequality in time yield
\[
\|\Psi(Y)-\Psi(Y')\|_E^2
\le
C_\varepsilon\left(T'^2+T'^3\right)
\|Y-Y'\|_E^2.
\]
Hence, for \(T'>0\) sufficiently small,
\[
\|\Psi(Y)-\Psi(Y')\|_E
\le
q\|Y-Y'\|_E,
\qquad q<1,
\]
so that \(\Psi\) is a contraction on \(E\).

By the Banach--Caccioppoli fixed-point theorem, there exists a unique fixed point
\(X_\varepsilon\in E\) such that
\[
\Psi(X_\varepsilon)=X_\varepsilon.
\]
This yields a unique strong solution on \([0,T']\). Since the Lipschitz constants are uniform in time, the same argument can be iterated on finitely many successive subintervals, 
using \(X^\varepsilon_{T'}\) and \(c_\varepsilon(T',\cdot)\) as the initial state; equivalently, the accumulated history up to \(T'\) is fixed.
Therefore, the regularised system admits a unique strong solution on the whole interval \([0,T]\).

\emph{Well-posedness of the unsmoothed system.} 
As in \cite[Theorem~3.3]{rui2025LJ}, let
\(U\subset\mathbb R^{\widetilde Nd}\) denote the set of
non-overlapping configurations,
\begin{equation*}
  U := \mathbb{R}^{\widetilde Nd} \backslash \left\{\bigcup_{1\leq j<k\leq \widetilde N} \left\{x = (x^{(1)},...,x^{(\widetilde N)})\in \mathbb{R}^{\widetilde Nd} : x^{(j)}= x^{(k)} \right\}\right\}.
\end{equation*} 
Furthermore, the particle system 
\begin{equation*} 
dX_t^{i} =\left[-\nabla V * \widetilde{\nu}_t^{\widetilde N} + \widetilde{G}\left[ \mathcal{D}\left(c_0(\cdot), \int_0^t \widetilde{u}_{\widetilde N}(s,\cdot) ds\right) \right] \right](X_t^{i})\,dt 
+ \sigma\,dW_t^i
\end{equation*}

can be viewed as an SDE on \(\mathbb R^{\widetilde N d}\) whose drift consists of the singular gradient part \(-\nabla\Phi\) and an additional path-dependent perturbation. The total interaction potential is
\[
\Phi(x)
=
\frac{1}{N}
\sum_{1\le i<j\le \widetilde N}
V(x^i-x^j),
\qquad
x=(x^1,\dots,x^{\widetilde N}).
\]
Choose \(\delta_{\widetilde N}\ge0\) such that
\(\Phi+\delta_{\widetilde N}\ge0\) on \(U\).

The environmental contribution is defined componentwise by
\[
\mu_t^i
=
\widetilde G\left[
\mathcal D\left(
c_0(\cdot),
\int_0^t \widetilde{u}_{\widetilde N}(s,\cdot)\,ds
\right)
\right](X_t^i),
\qquad
i=1,\ldots,\widetilde N.
\]
The environmental drift is progressively measurable and uniformly
bounded by Proposition~\ref{prop:regularity_G}. We therefore adapt the
non-attainment argument used in the proof of
\cite[Theorem~3.3]{rui2025LJ}.

For convenience, for \(i\neq j\) we write
\[
V_{ij}:=V(x^i-x^j),
\qquad
F_{ij}:=-\nabla V(x^i-x^j).
\]
We now adapt the non-attainment argument used in the proof of \cite[Theorem~3.3]{rui2025LJ}. Although the environmental
perturbation is path-dependent, it is progressively measurable and uniformly bounded by Proposition~\ref{prop:regularity_G}; these are
the properties needed in the Lyapunov estimate below.
It remains to verify the crucial condition $H_3$.
The corresponding coercivity estimate for the pure Lennard--Jones
system is established in the proof of
\cite[Theorem~3.3]{rui2025LJ}. In particular, one has
\begin{equation}\label{estimate:LJ_chap5}
\begin{split}
    |\nabla\Phi(x)|^2 -    \frac{\sigma^2}{2}\Delta\Phi(x) 
    &\ge \frac{c_{\widetilde N}}{N^2} \sum_{1\leq i <j \leq N}\|F_{ij}\|^2 - \frac{\sigma^2}{N}\sum_{1\leq i <j \leq N}\Delta V_{ij} - \frac{C_{\widetilde N}}{N^2}\\
    &= \sum_{1\leq i <j \leq N}\left( \frac{c_{\widetilde N}}{N^2}\|F_{ij}\|^2 - \frac{\sigma^2}{N}\Delta V_{ij}\right)  - \frac{C_{\widetilde N}}{N^2}
    \end{split}
\end{equation}
In our setting, the additional term $\mu\cdot\nabla\Phi$ is controlled by the uniform boundedness of the environmental drift. More precisely, using \eqref{eq:G_bounded} and Young's inequality
\[
-\mu_t(x)\cdot\nabla\Phi(x)
\ge
-\frac12|\nabla\Phi(x)|^2-C,
\]
for some constant \(C>0\). Hence, using \eqref{estimate:LJ_chap5} together with
\[
\Delta\Phi(x)=\frac{2}{N}\sum_{1\leq i <j \leq N}\Delta V_{ij},
\]
we obtain
\begin{equation*}
\begin{split}
|\nabla\Phi(x)|^2
-\frac{\sigma^2}{2}\Delta\Phi(x)
-\mu_t(x)\cdot\nabla\Phi(x)
&\ge
\frac12|\nabla\Phi(x)|^2
-\frac{\sigma^2}{2}\Delta\Phi(x)-C
\\
&\ge
\frac{c_{\widetilde N}}{2N^2}\sum_{1\leq i <j \leq N}\|F_{ij}\|^2
-\frac{\sigma^2}{N}\sum_{1\leq i <j \leq N}\Delta V_{ij}
-C',
\end{split}
\end{equation*}
for a suitable constant \(C'>0\).

As in the proof of \cite[Theorem~3.3]{rui2025LJ}, the squared Lennard--Jones force provides the dominant contribution near the singular set. Indeed, as \(|x^i-x^j|\to0\),
\[
\|F_{ij}\|^2\sim C|x^i-x^j|^{-2(\alpha+1)},
\]
whereas
\[
\Delta V_{ij}=O\!\left(|x^i-x^j|^{-(\alpha+2)}\right).
\]
Since \(2(\alpha+1)>\alpha+2\), the positive squared-force term dominates the Laplacian contribution near collisions. Away from the singular set, both terms are bounded. Consequently, the right-hand side of the previous inequality is bounded from below on the configuration space, and hence condition \(H_3\) is satisfied.

We now adapt the non-attainment argument used in the proof of\cite[Theorem~3.3]{rui2025LJ}. For \(\varepsilon>0\), set
\[
    U_\varepsilon
    :=
    \{x\in U:d(x,\partial U)>\varepsilon\}.
\]
For the regularised solution
\(X^\varepsilon\), define
\[
\tau_\varepsilon
:=
\inf\{t\in[0,T]:X_t^\varepsilon\notin U_\varepsilon\},
\qquad
\inf\emptyset:=+\infty.
\]
The same non-attainment argument applies here, since the additional
progressively measurable perturbation satisfies the Lyapunov bound
verified above, together with the required stopped It\^o and martingale
estimates. Hence
\begin{equation}\label{eq:nonattainment_chap_reacting}
    \lim_{\varepsilon\downarrow0}
    \mathbb P(\tau_\varepsilon\le T)=0.
\end{equation}
Let \(0<\varepsilon'<\varepsilon\). On
\(\{X_0\in U_\varepsilon\}\), the two regularisations coincide with the
original Lennard--Jones interaction on \(U_\varepsilon\). Since the
environmental drift is the same functional of the particle trajectory
and the equations are driven by the same Brownian motions, local pathwise uniqueness gives
\[
X_t^\varepsilon=X_t^{\varepsilon'},
\qquad
0\leq t\leq T\wedge \tau_\varepsilon.
\]
If \(X_0\notin U_\varepsilon\), then \(\tau_\varepsilon=0\). Hence
\[
\tau_\varepsilon\leq\tau_{\varepsilon'}
\qquad
\text{for }0<\varepsilon'<\varepsilon.
\]
Let \(\varepsilon_n\downarrow0\). By
\eqref{eq:nonattainment_chap_reacting} and the monotonicity above,
\[
\mathbb P\left(
\lim_{n\to\infty}\tau_{\varepsilon_n}>T
\right)=1.
\]
Thus, almost surely, \(\tau_{\varepsilon_n}>T\) for all sufficiently
large \(n\). By consistency, the corresponding processes
\(X^{\varepsilon_n}\) are then eventually identical on \([0,T]\).
Their common value defines a process \(X\), while the associated fields
\(c_{\varepsilon_n}\), determined by \eqref{eq:c_epsilon}, are eventually
identical as well. Their common value \(c\) yields, together with \(X\),
a solution of \eqref{eq:SDE_LJcg_nojumps} on \([0,T]\).

Finally, pathwise uniqueness follows by localisation: two solutions
driven by the same Brownian motions coincide up to each corresponding
exit time, while the non-attainment estimate implies that these exit
times eventually exceed \(T\).

Finally, let \(\tau\le T\) be a bounded stopping time. The localised
Lyapunov estimate gives
\[
    \mathbb E[\Phi(X_\tau)+\delta_{\widetilde N}]
    +
    \frac12\mathbb E\int_0^\tau|\nabla\Phi(X_s)|^2\,ds
    \le
    \mathbb E[\Phi(X_0)+\delta_{\widetilde N}]+CT.
\]
Since \(|V|\le V+2\delta'\), this estimate and the integral form of the
SDE imply that \(X_\tau\) satisfies \eqref{eq:init_finite}. Moreover,
\(c(\tau,\cdot)\) is \(\mathcal F_\tau\)-measurable and, by its
explicit representation,
\[
    0\le c(\tau,x)\le c_0(x)\le m_0(x),
    \qquad
    \mathbb E\|c(\tau,\cdot)\|_{L^2}^2
    \le
    \mathbb E\|c_0\|_{L^2}^2<\infty.
\]
Thus the hypotheses of the theorem are preserved at bounded stopping
times.

Moreover, for \(t\ge\tau\),
\[
c(t,x)
=
c(\tau,x)
\exp\left(
-\lambda\int_\tau^t
\widetilde u_{\widetilde N}(s,x)\,ds
\right),
\]
so that, after shifting time, the system has the same form.

This concludes the proof.
\end{proof}

\section{Well-posedness of the full system with particle removal}
\label{sec:final}

Now we discuss the main result, that is the construction of a global strong solution to the full coupled system \eqref{eq:full_system}.
The proof relies on an interlacing procedure, originally introduced by \cite{Applebaum} for Lévy processes and subsequently used successfully in more general contexts; see, e.g., \cite{HamblyJettkant2025,XiZhu2017}. 
This approach allows for a pathwise construction of a unique solution, obtained by concatenating strong solutions of the diffusion system between consecutive random reaction times.

In the present framework, each particle can undergo at most one reaction, after which it is sent to the cemetery state and no longer contributes to the dynamics. Then the interlacing procedure terminates after finitely many steps, yielding a global pathwise strong solution to the full system.

\begin{theorem}[Global well-posedness with particle removal]
Let \(d\ge1\), \(N\ge2\), and \(T>0\). Assume that
\(X_0=(X_0^1,\dots,X_0^N)\) is \(\mathcal F_0\)-measurable,
collision-free almost surely, and satisfies
\[
    \mathbb E|X_0|^2<\infty,
    \qquad
    \mathbb E\, \Big[
        \sum_{1\le i<j\le N}
        |V(X_0^i-X_0^j)|
    \Big]<\infty.
\]
Let
\(c_0:\Omega\times\mathbb R^d\to\mathbb R_+\)
be \(\mathcal F_0\otimes\mathcal B(\mathbb R^d)\)-measurable, with
\[
    \mathbb E\|c_0\|_{L^2(\mathbb R^d)}^2<\infty,
    \qquad
    0\le c_0(\omega,x)\le m_0(x)
\]
for almost every \(\omega\) and every \(x\in\mathbb R^d\).
For notational convenience, assume \(H_0=(0,\dots,0)\).

Then \eqref{eq:full_system} admits a pathwise unique global strong
solution on \([0,T]\).
\end{theorem}
The assumption \(H_0=(0,\dots,0)\) is only for notational convenience:
the same construction applies if some particles have already reacted,
by starting from the corresponding active subconfiguration.
\begin{proof}
The proof proceeds by an interlacing construction along the successive reaction times.
We first prove the existence of a strong solution by induction. We consider three main steps.

\smallskip
\noindent\emph{First switching time and particle selection.} For each particle $i\in N^*=\{1,\dots,N\}$, the killing times $\tau_i$  \eqref{eq:tau_def} can be written as the first acceptance time for particle $i$ by
\[
\tau_i:=\inf\Bigl\{t>0:\ \int_0^t \int_0^\infty 
\mathbbm 1_{\{H^i_{s-}=0\}}\mathbbm 1_{\{z\le \lambda_i(s)\}}
\,M^i(ds,dz)\ge 1\Bigr\},
\]
with the convention $\inf\emptyset:=+\infty$. Since the reaction intensities are bounded and the Poisson random measures \(M^1,\ldots,M^N\) are independent with diffuse time intensity, the accepted reaction processes almost surely have no simultaneous jumps. Hence, if we define the first switching time of the discrete process $H$ by
\[
\tau(1):=\inf\{t>0:\ H_t\neq H_0\}=\min_{1\le i\le N}\tau_i\qquad\text{a.s.}
\]
it follows that on the event $\tau(1)\le T $ almost surely the minimum $
\tau(1)$ 
is attained by a unique index. We denote this index by $i_1$, so that 
$\tau(1)=\tau_{i_1}$.

\smallskip
\noindent\emph{Evolution up to the first switch.}
Set \(h_0:=H_0=(0,\dots,0)\).
Then, Theorem~\ref{thm:finite_wp} yields the unique fixed-active solution \((X^{h_0},c^{h_0})\) on \([0,T]\). Using this solution in the state-dependent reaction intensities, we define the first accepted reaction time \(\tau(1)\).
We therefore set
\[
(X_t,H_t, c(t,\cdot)) := \bigl(X^{h_0}(t),h_0,c^{h_0}(t,\cdot)\bigr),
\qquad t\in[0,\tau(1)).
\]
If $\tau(1)\ge T$, the construction stops. Otherwise, at time $\tau(1)$, the particle $i_1$ is removed from the dynamics and sent to the cemetery state, and we update the activity vector $h_0$, switching its $i_1$-th component to $1$:
\[
H_{\tau(1)}=h_1 = (h_0^1,\dots,h_0^{i_1-1},1,h_0^{i_1+1},\dots,h_0^N).
\]
The remaining $N-1$ particles retain their positions at time $\tau(1)$ and evolve according to the same dynamics, now interpreted as an $(N-1)$-particle system coupled with the field $c$, whose evolution is governed by \eqref{eq:full_system} with the empirical measure formed from the surviving particles.

By the stopping-time stability established in
Theorem~\ref{thm:finite_wp}, the surviving particle configuration and
the field \(c(\tau(1),\cdot)\) satisfy the hypotheses of the same
theorem. Hence the \((N-1)\)-particle system admits a unique strong
solution on \([\tau(1),\tau(2))\), where \(\tau(2)\) is the next
reaction time. We follow this reasoning and proceed by induction.

\smallskip
\noindent\emph{Iteration.}
Assume that the solution $(X,H,c)$ has been constructed up to time $\tau(n)$ with $H_{\tau(n)}=h_n$. 
Starting from the state
\[
(X_{\tau(n)}, H_{\tau(n)},c(\tau(n),\cdot))
\]
By Theorem~\ref{thm:finite_wp} , there exists a unique strong
solution with \(N-n\) active particles starting from
\(\tau(n)\).
Define \(H\) on \((\tau(n),T]\) by the same thinning rule, using the
already-fixed family of Poisson random measures \((M^i)_{i=1}^N\), and define
\begin{equation}\label{def:taun_iterative}
  \tau(n+1):=\inf\{t>\tau(n):\ H_t\neq H_{\tau(n)}\}.
\end{equation}
If $\tau(n+1)<T$, let $i_{n+1}$ be the (a.s.\ unique) index that jumps at time $\tau(n+1)$, and set 
\[
H_{\tau(n+1)} := h_{n+1} = (h_n^1,\dots,h_n^{i_{n+1}-1},1,h_n^{i_{n+1}+1},\dots,h_n^N).
\]
On the interval $I_n := [\tau(n),\tau(n+1))$, we define $(X,H,c)$ to coincide with the solution corresponding to the fixed activity vector $h_n$
\[
(X_t,H_t, c(t,\cdot)) := \bigl(X^{h_n}(t),h_n,c^{h_n}(t,\cdot)\bigr),
\qquad t\in[\tau(n),\tau(n+1)).
\]
After the $N$th removal, we explicitly define
\[
H_t=(1,\dots,1),\quad X^i_t=\Delta, \quad \nu^N_t=0,\quad 
\]
and \(c(t,\cdot)=c(\tau(N),\cdot),\quad t \ge \tau(N). \)
The process $H$ has at most $N$ jumps: the interlacing times $\tau(n)$ cannot accumulate before $T$, and the concatenated process is well-defined on the whole interval $[0,T]$ after finitely many steps. This completes the interlacing construction of a strong global solution.

In order to prove pathwise uniqueness, let $Z:=(X,H,c)$ and $\bar Z:= (\bar X,\bar H,\bar c)$ be two strong solutions of
\eqref{eq:full_system} on $[0,T]$, defined on the same filtered probability
space, driven by the same Brownian motion $W$ and the same family of Poisson random measures
$(M^i)_{i=1}^N$, with a.s. identical initial data, $Z_0=\bar Z_0$.

Define the first disagreement time of the discrete components by
\[
\zeta := \inf\{t>0: H_t\neq \bar H_t\}, \qquad \inf \emptyset := +\infty.
\]
By definition,
\begin{equation}\label{eq:H_agree_before_zeta}
H_t=\bar H_t,\qquad\text{for all }t<\zeta.
\end{equation}

\noindent\emph{Interlacing before $\zeta$.}
Let $\tau(n)$ be the jump times defined recursively in \eqref{def:taun_iterative}, with $\tau(0):=0$. 
 Since $H$ takes values in $\mathcal H=\{0,1\}^N$ and each coordinate can jump at most once, there are at most $N$ strict jumps.
For each $n$, on the interval $I_n:=[\tau(n),\tau(n+1))$ we have
\[
H_t\equiv h_n:=H(\tau(n)).
\]
Because of \eqref{eq:H_agree_before_zeta}, the same interlacing times describe $ \bar H$ up to
$\zeta$, and for all $t<\zeta$ with $t\in I_n$,
\[
\bar H(t)=H(t)=h_n.
\]
Hence, on each interval $I_n\cap[0,\zeta)$, both $Z$ and $\bar Z$ solve the same subsystem with
constant activity vector $h_n$, driven by the same Brownian motions and with the same initial
conditions at time $\tau(n)$.
By pathwise uniqueness for the corresponding diffusion system
(Theorem~\ref{thm:finite_wp}), we obtain
\begin{equation}\label{eq:agree_before_zeta_marked}
Z(t)=\bar Z(t)\qquad\text{for all }t<\zeta,\ \text{a.s.}
\end{equation}

\noindent\emph{No disagreement of $H$ at $\zeta$.}
It remains to exclude the possibility of a disagreement of $H$ at time $\zeta$. Both jump processes are defined via the same random measure $M$, as in \eqref{eq:full_system}, and at each switching time the jumping coordinate is determined by the unique accepted mark. For $i=1,\dots,N$, define the acceptance indicators
\[
I_i(s,z)
:=\mathbbm{1}_{\{H_{s-}^i=0\}}
\mathbbm{1}_{\{z\le\widetilde \lambda\,c(s,X_{s-}^i)\}},
\]
\[
\bar I_i(s,z)
:=\mathbbm{1}_{\{\bar H_{s-}^i=0\}}
\mathbbm{1}_{\{z\le\widetilde \lambda\,\bar c(s,\bar X_{s-}^i)\}}.
\]
Let
\[
J_t
:=
\sum_{i=1}^N
\int_{(0,t\wedge\zeta]}
\int_0^\infty
\mathbbm{1}_{\{I_i(s,z)\neq \bar I_i(s,z)\}}
\,M^i(ds,dz).
\]
The random variable $J_t$ counts the number of Poisson jumps up to time $t\wedge\zeta$ for which one
system accepts the mark while the other rejects it.

If $\zeta\le t$, then there exists a first Poisson mark at which the two systems differ, and hence
$J_t\ge1$. Therefore,
\[
\mathbbm{1}_{\{\zeta\le t\}}\le J_t
\qquad\Rightarrow\qquad
\mathbb P(\zeta\le t)\le \mathbb E[J_t].
\]
Since the integrand in the definition of $J_t$ is predictable, the compensation formula for random measures yields
\[
\mathbb E[J_t]
=
\mathbb E\!\left[
\sum_{i=1}^N
\int_0^{t\wedge\zeta}
\int_0^\infty
\mathbbm{1}_{\{I_i(s,z)\neq \bar I_i(s,z)\}}
\,dz\,ds
\right].
\]
For all $s<\zeta$, by \eqref{eq:H_agree_before_zeta} and \eqref{eq:agree_before_zeta_marked} we have that almost surely
$H(s-)=\bar H(s-), X_{s-}=\bar X_{s-}, c(s,\cdot)=\bar c(s,\cdot).$ 
Hence, for all $i,z$, $I_i(s,z)=\bar I_i(s,z)$, {a.s.}
and the integrand vanishes identically.
Consequently, for all $t\in[0,T]$ we get $\mathbb E[J_t]=0$, and then
\[
\mathbb P(\zeta\le t)=0.
\]
We conclude that $\mathbb P(\zeta\le T)=0$, so $H=\bar H$ almost surely on $[0,T]$.
Together with \eqref{eq:agree_before_zeta_marked}, this implies
\[
(X,H,c)=(\bar X,\bar H,\bar c)
\quad\text{a.s.\ on }[0,T],
\]
which proves pathwise uniqueness of the global strong solution to
\eqref{eq:full_system}.
\end{proof}

\printbibliography

\end{document}